\documentclass[a4paper, oneside]{amsart}

\usepackage{geometry}
\usepackage{tgpagella}
\usepackage{layouts}
\usepackage[dvipsnames]{xcolor} 
\usepackage[utf8]{inputenc}
\usepackage[T1]{fontenc}
\usepackage{siunitx}
\usepackage{amsfonts,amssymb}
\usepackage[hidelinks]{hyperref}
\usepackage{amsthm,cleveref}
\usepackage{color}
\usepackage{graphicx}
\usepackage{stmaryrd}
\usepackage{bm}
\usepackage{caption}
\usepackage[position=t]{subcaption}

\usepackage[color=black!10!white]{todonotes}
\usepackage{mathtools}
\usepackage{tikz}
\usetikzlibrary{colorbrewer, angles, calc, quotes,decorations,patterns}
\usepackage{pgfplots}
\usepackage{pgfplotstable}
\usepackage{booktabs}
\usepackage{placeins}
\usepackage{csquotes}

\usepgfplotslibrary{colorbrewer}
\pgfplotsset{compat=1.13,
  axis lines=left,
legend style={draw=none},
scaled ticks=false
}
\definecolor{maincolor}{gray}{0}

\usepackage{microtype}
\usepackage{enumitem}

\newcommand{\fullnorm}[2][]{\left\vert\kern-0.25ex\left\vert\kern-0.25ex\left\vert #2 
      \right\vert\kern-0.25ex\right\vert\kern-0.25ex\right\vert_{\mathrm{DG} #1}}
\newcommand{\fullnormfixedsize}[2][]{\vert\kern-0.25ex\vert\kern-0.25ex\vert #2 
      \vert\kern-0.25ex\vert\kern-0.25ex\vert_{\mathrm{DG} #1}}

\newtheorem{theorem}{Theorem}

\newtheorem{lemma}{Lemma}
\newtheorem{remark}{Remark}

\newtheorem*{ind-assumption}{Induction Assumption}

\newtheorem{proposition}[lemma]{Proposition}
\newtheorem{corollary}[lemma]{Corollary}

\let\epsilon\varepsilon

\let\phi\varphi

\let\theta\vartheta

\newcommand{\spacenameinnorm}{\mathcal{D}}
\newcommand{\spacenameinnormhom}{\mathcal{G}}

\newcommand{\wnormdg}[5]{\left\vert\kern-0.25ex\left\vert\kern-0.25ex\left\vert #1 
      \right\vert\kern-0.25ex\right\vert\kern-0.25ex\right\vert_{\spacenameinnorm^{#3}_{#4}(#5)} }
\newcommand{\wnormdghom}[5]{\left\vert\kern-0.25ex\left\vert\kern-0.25ex\left\vert #1 
      \right\vert\kern-0.25ex\right\vert\kern-0.25ex\right\vert_{\spacenameinnormhom^{#3}_{#4}(#5)} }

\newcommand{\phat}{{p_\star}}
\newcommand{\ghat}{\gamma_\star}
\newcommand{\thetahat}{\theta_\star}
\newcommand{\Cinterp}{C_\mathrm{interp}}
\newcommand{\rhotilde}{\tilde{\rho}}
\newcommand{\ratio}{\mathfrak{f}}
\newcommand{\rationum}{\cm{\frac{1 + \sqrt{5}}{2}}}
\newcommand{\CSp}{C_{S, p}}
\newcommand{\CCp}{C_{2, p}}
\newcommand{\CCCp}{C_{3, p}}
\newcommand{\CCCpstar}{C_{3, \phat}}
\newcommand{\CSq}{C_{S, \frac{3p}{2}}}
\newcommand{\Cregp}{C_{\mathrm{reg},p}}
\newcommand{\Cregpstar}{C_{\mathrm{reg},\phat}}
\newcommand{\Cregq}{C_{\mathrm{reg}, \frac{3p}{2}}}
\newcommand{\Cregqstar}{C_{\mathrm{reg}, \frac{3\phat}{2}}}
\newcommand{\Cphi}{C_{\Phi}}
\newcommand{\Aphi}{A_{\Phi}}

\newcommand{\cK}{\mathcal{K}}
\newcommand{\cJ}{\mathcal{J}}

\newcommand{\fc}{{\mathfrak{c}}}
\newcommand{\fC}{{\mathfrak{C}}}

\newcommand{\dalpha}{\partial^\alpha}

\newcommand{\dbeta}{\partial^\beta}
\newcommand{\dab}{\partial^{\alpha+\beta}}
\newcommand{\betam}{{|\beta|}}

\newcommand{\alpham}{{|\alpha|}}
\newcommand{\xim}{{|\xi|}}

\newcommand{\hB}{\widehat{B}}
\newcommand{\hr}{\hat{r}}
\newcommand{\Ctil}{{\widetilde{C}}}
\newcommand{\Atil}{{\widetilde{A}}}
\newcommand{\tgamma}{{\tilde{\gamma}}}

\newcommand{\tB}{{\widetilde{B}}}

\graphicspath{{./figures/}}
\pgfplotstableset{
    columns/N/.style={int detect,column type=r,column name=$N$},
    columns/eDG/.style={
	sci sep align, precision=2,
	column name=$\|u-u_\delta\|_{\mathrm{DG}}$,
      },
    columns/eL2/.style={
	sci sep align, precision=2,
	column name=$\|u-u_\delta\|_{L^2(\Omega)}$,
      },
    columns/eLinf/.style={
	sci sep align, precision=2,
	column name=$\|u-u_\delta\|_{L^\infty(\Omega)}$,
    },
    columns/elambda/.style={
	sci sep align, precision=2,
        column name=$|\lambda - \lambda_\delta|$,
      },
    columns/bL2/.style={
	sci sep align, precision=2,
	column name=$b_{L^2}$,
      },
    columns/bDG/.style={
	sci sep align, precision=2,
	column name=$b_{\mathrm{DG}}$,
    },
    columns/bLinf/.style={
	sci sep align, precision=2,
	column name=$b_{L^\infty}$,
    },
    columns/blambda/.style={
	sci sep align, precision=2,
	column name=$b_{\lambda}$,
    },
    columns/slope/.style={
	sci sep align, precision=4,
	column name=$\mathfrak{s}$,
    },
    empty cells with={--}, 
    every head row/.style={before row=\toprule,after row=\midrule},
    every last row/.style={after row=\bottomrule}
  }

\usepackage[foot]{amsaddr}
\newcommand{\cm}[1]{{#1}}

\author{Yvon Maday$^{\dagger, *}$}
\address[$\dagger$]{Sorbonne Université, CNRS, Université de Paris, Laboratoire
Jacques-Louis Lions (LJLL), F-75005 Paris, France. }
\email{yvon.maday@ann.jussieu.fr}
\author{Carlo Marcati$^\diamond$}
\address[$\diamond$]{Seminar for Applied Mathematics, ETH Zürich, Rämistrasse 101,
CH-8092 Zurich, Switzerland}
\email{carlo.marcati@sam.math.ethz.ch}
\thanks{$^*$Yvon Maday acknowledges funding from the European Research Council (ERC) under the European Union's Horizon 2020 research and innovation programme (grant agreement No 810367).}

\makeatletter
\@namedef{subjclassname@2020}{\textup{2020} Mathematics Subject Classification}
\makeatother
\title{Weighted analyticity of Hartree-Fock eigenfunctions}
\begin{document}
\maketitle
\begin{abstract}
  We prove analytic-type estimates in weighted Sobolev spaces on the eigenfunctions of a class
of elliptic and nonlinear eigenvalue problems with singular potentials, which includes the Hartree-Fock equations.
Going beyond classical results on the analyticity of the wavefunctions away from
the nuclei, we prove weighted estimates locally at each singular point, with
precise control of the derivatives of all orders. 

Our estimates have far-reaching consequences for the approximation of the
eigenfunctions of the problems considered, and they can be used to prove a
priori estimates on the numerical solution of such eigenvalue problems.
\end{abstract}
\section{Introduction}
\label{sec:intro}

The Hartree-Fock equations are one of the most studied and used models in \textit{ab initio}
quantum chemistry in order to approximate the behavior of many-body quantum
system \cite{Szabo2012}. Due to their (relative) simplicity, they constitute
a starting point both for the analysis and for the 
computation of the state of many complex systems. The precise characterization
of their solutions is therefore a subject of great theoretical and practical interest.

In this paper, we prove analytic-type estimates in weighted Sobolev spaces on the wave functions of a class
of elliptic, nonlinear systems, which includes the Hartree-Fock model.
Specifically, we consider operators that contain potentials that are singular
(divergent) at
a set of isolated points (physically, the locations of the nuclei) in $\mathbb{R}^d$, $d\in \{2,3\}$, but that are regular
otherwise. Due to the presence of these singularities, the eigenfunctions will
not, in general, be regular in classical Sobolev spaces and are well known
\cite{Kato1957} to exhibit cusps at the point singularities.
The regularity of functions with point singularities is better described in the
context of weighted Sobolev spaces, in which higher order derivatives are
multiplied by a weight representing the distance from the singularity. In these
spaces, under some assumptions on the potential, we can therefore derive
analytic-type bounds on the growth of the norms of the eigenfunctions of the
nonlinear elliptic systems under consideration. Essentially, we refine the
known result on analyticity of the wavefunctions away from the nuclei (see,
e.g., \cite{Fournais2002,Lewin2004}) and show how the radius of convergence of
Taylor series associated to the wavefunction decreases to zero in the vicinity
of the singular points.

The theory of weighted Sobolev spaces of the kind we consider here has its roots
in the analysis of elliptic problems in non smooth domains and was initiated in
the second half of the twentieth century \cite{Kondratev1967}. 
Analytic regularity of solutions to linear elliptic systems in polygons and polyhedra
has been analyzed, e.g, in \cite{Guo2006, Costabel2012}. Concerning nonlinear
problems, we mention our work on nonlinear Schrödinger equations
\cite{Maday2019b} and on the Navier-Stokes equation in plane polygons \cite{Marcati2019}.
For a general theory of elliptic regularity in weighted spaces, we refer the
reader to, e.g., \cite{Grisvard1985, Kozlov1997, Kozlov2001, Mazya2010}, and the
recent work \cite{Dahlke2019}.
Here, we try to make our exposition as independent as possible from the usual
notation of weighted, Kondratev-type Sobolev spaces and introduce them only in
the appendix. The theory of regularity in those spaces is, nonetheless,
ultimately central to the derivation of our estimates.

The techniques used in the present paper
are heavily inspired by those used in 
\cite{DallAcqua2012} to prove analyticity away from the nuclei of the solution to the relativistic
Hartree-Fock equations. Here, we transport those techniques in a weighted
framework, and use them to estimate higher order norms of the nonlinear terms. 
The analysis of linear, many-body Schrödinger-type operators has been carried
out, among others, in \cite{Ammann2012, Fournais2018} in a functional setting
very similar to the one considered here. An asymptotic analysis at the
nuclei for the
Hartree-Fock equation with Coulomb potential is carried out, with different
tools, in \cite{Flad2008}; the electron-electron singularities emerging in many-body models are
analyzed in \cite{Flad2011, Flad2015}.
Here, we only consider two and three dimensional nonlinear models with isolated point
singularities; we furthermore take into account a wider class of potentials than Coulomb
ones, as we allow for more general weighted analytic potentials.
The technique used in this paper can also be rather directly extended to deal
with the nonlinear part of other types of operators, once the behavior of the
linear part of the operator is well understood: see, for example, the
application to Navier-Stokes equations in \cite{Marcati2019}.

We will discuss, in the next subsection, some
of the consequences of the weighted analytic regularity of the eigenfunctions,
in particular from the point of view of their numerical approximation, through
linear and nonlinear techniques.
Then, after having clarified our notation, we shortly introduce the Hartree-Fock equations and the more
general nonlinear elliptic system, in Section \ref{sec:HF}. In the following Section \ref{sec:HF-analytic},
specifically in Theorem \ref{theorem:HFreg}, we introduce the main
result of this paper, and most of the section will be devoted to its proof.
We conclude by introducing, for the sake of completeness, the definition of
weighted, homogeneous and non homogeneous, Sobolev spaces
and some technical results, in Appendix \ref{appendix:technical}.
\subsection{Consequences of weighted analytic regularity}
The weighted analytic regularity of the solutions to Hartree-Fock and more
general elliptic problems has important and far-reaching consequences for the
numerical solution of those problems.
We can, indeed, obtain exponential rates of convergence of solutions obtained via
numerical methods based on finite elements, see \cite{Schotzau2013a,Schotzau2013b} for a general
approximation theory and \cite{Maday2019a,Maday2019b,Heid2019} for applications
to linear and nonlinear eigenvalue problems, and on virtual elements \cite{Certik2020}.
In addition, nonlinear approximation techniques based on tensor compression and
on the solution of partial differential equation in tensor-formatted form
also provide exponentially convergent solutions to problems with weighted
analytic solutions \cite{MRS2019}. Similarly, for such functions, neural networks with ReLU
activation function can be constructed so that their size is bounded
polylogarithmically with respect to the error (or, equivalently, the error
converges exponentially with respect to the size) \cite{MPOS2020}. The present
analytic-type regularity results, therefore, allow for an \textit{a priori}
analysis of multiple numerical methods which have proven and will probably prove
useful for applications.
\subsection{Notation}
Let the space dimension be $d\in\{2,3\}$. We denote by $\mathbb{N}$ the
set of positive integers, with $\mathbb{N}_0 = \{0\} \cup \mathbb{N}$. For $k\in
\mathbb{N}$ and $1\leq p \leq \infty$,
Sobolev spaces are denoted by $W^{k,p}$, with their Hilbertian version written
$H^{k} = W^{k,2}$.
 For two multi indices $\alpha = (\alpha_1, \dots,
\alpha_d)\in \mathbb{N}^d$ and $\beta = (\beta_1, \dots,
\beta_d)\in\mathbb{N}^d$, we write $\alpham = \sum_i\alpha_i$, $\alpha! =\alpha_1!\cdots \alpha_d!$, $\alpha
+ \beta = (\alpha_1+\beta_1, \dots, \alpha_d+ \beta_d)$, and
\begin{equation}
  \label{eq:binom}
  \binom{\alpha}{\beta} = \frac{\alpha!}{\beta! (\alpha-\beta)!}.
\end{equation}
 We recall from \cite{Kato1996} that
\begin{equation*}
  \sum_{\substack{\betam = n\\\beta\leq\alpha}} \binom{\alpha}{\beta} = \binom{\alpham}{n}.
\end{equation*}
Let $x = (x_1, \dots, x_d)$: we indicate by $\partial_i$ the partial derivative
with respect to $x_i$, and for $\alpha = (\alpha_1, \dots, \alpha_d)\in
\mathbb{N}_0^d$, $\dalpha = \partial_1^{\alpha_1}\cdots \partial_d^{\alpha_d}$.

\section{The Hartree-Fock equations}
\label{sec:HF}
Let $N, N_n\in \mathbb{N}$ be the number of electrons and nuclei of a
  system, let $\fc_i$, $i=1, \dots, N_n$ be isolated points in
$\mathbb{R}^3$ representing the positions of the nuclei,
  and let $Z_i>0$ be the charges of the nuclei, for all $i=1, \dots, N_n$.
the Hartree-Fock problem consists in finding the smallest eigenvalues
$\lambda_\iota$ and
associated orthonormal eigenfunctions $\phi_\iota$, $\iota=1, \dots, N$ of the equations
\begin{equation}
  \label{eq:HF-integral}
  -\frac{1}{2} \Delta \phi_\iota + V_C\phi_\iota + \left( \rho_\Phi \star \frac{1}{|\cdot|} \right) \phi_\iota - \int_{\mathbb{R}^3} \frac{\tau_\Phi(\cdot,y)}{|\cdot-y|} \phi_\iota(y) dy  = \lambda_\iota \phi_\iota\qquad \iota=1,\dots, N\qquad \text{ in }\mathbb{R}^3
\end{equation}
where $V_C$ is the potential
\begin{equation*}
  V_C(x) = -\sum_{i=1}^{N_n} \frac{Z_i}{|x-\fc_i|}, 
\end{equation*}
and
\begin{equation*}
    \tau_\Phi(x,y)= \sum_{\iota=1}^{N} \phi_\iota(x)\phi_\iota(y),  \qquad \rho_\Phi(x) = \tau_\Phi (x,x).
\end{equation*}
The analyticity of the wave functions away from the positions of the nuclei
(i.e., the singularities of $V$) is classical, see, e.g., \cite{Fournais2002,Lewin2004}. In
this setting we consider instead the parts of the domain containing the nuclei,
in order to deduce the weighted estimates.

Let now $V:\mathbb{R}^d\to \mathbb{R}$ be a potential to be specified later;
  we consider the nonlinear elliptic system given by
\begin{equation}
  \begin{aligned}
  \label{eq:HF}
   \left(-\Delta + V\right)\phi_{\iota} + \sum_{\sigma,a, b=1}^N c^{\iota\sigma}_{ab} u_{ab} \phi_\sigma &= \lambda_{\iota} \phi_\iota  &\iota=1,\dots,N \\
-\Delta u_{ab} & = 4\pi \phi_a \phi_b  &a,b =1, \dots, N.
  \end{aligned}
\end{equation}
with $c^{\iota\sigma}_{ab}\in \mathbb{R}$ for all $\iota, \sigma, a, b =1,
\dots, N$ and $\lambda_{\iota}\in \mathbb{R}$ for all $\iota =1, \dots, N$.
The Hartree-Fock equations can be rewritten under the form \eqref{eq:HF},
  with $V=V_C$.
The nonlinear elliptic eigenvalue problem \eqref{eq:HF} is the one we
will analyze in the following.
\section{Weighted analyticity of eigenfunctions}
\label{sec:HF-analytic}
In this section, we present and prove our regularity result.
We will widen our scope from the Hartree-Fock equations and analyze the behavior
of the eigenfunctions near the singular points of the potential for solutions to
\eqref{eq:HF} in a $d$-dimensional domain for $d=2,3$. Our results, furthermore,
 will hold for a class of weighted analytic potential, including Coulomb
 potentials.

Given a set of isolated points $\fC$ in $\mathbb{R}^d$ such that there
  exists $D>0$ such that
\begin{equation}
  \label{eq:cond-fC}
  |\fc_i - \fc_j|\geq 4D >0\quad \forall \,\fc_i, \fc_j\in \fC,
\end{equation}
 we introduce the weight function $r : \mathbb{R}^d\to
\mathbb{R}$ such that
\begin{equation}
  \label{eq:r-def} 
r(x) = |x-\fc| \text{ in }B_{D}(\fc), \text{ for all }\fc\in \fC, \qquad r(x) \equiv 1 \text{ in }\left(\bigcup_{\fc\in\fC} B_{2D}(\fc)  \right)^C,
\end{equation}
and $r$ is smooth in $\mathbb{R}^d\setminus\fC$. The dependence
of $r$ in $x$ will be mostly omitted.

\begin{theorem}
    \label{theorem:HFreg}
    Let $\epsilon\in (0,1)$, $d\in\{2,3\}$, $r$ be defined as in \eqref{eq:r-def} for a
    collection of isolated points $\fC\subset \mathbb{R}^d$ such that
    \eqref{eq:cond-fC} holds and let $V$ be such that
    \begin{equation}
      \label{eq:V-hyp}
      \|r^{2-\epsilon + \alpham} \dalpha V\|_{L^\infty(\mathbb{R}^d)} \leq C_V A_V^{\alpham} \alpham!, \qquad \text{for all }\alpha\in \mathbb{N}_0^d,
    \end{equation}
    and that there exists a unique 
solution $\Phi = \{\phi_\iota\}_{\iota=1}^N \in (H^1(\mathbb{R}^d))^N$  to 
     \eqref{eq:HF}. Then, for any $\eta < \epsilon$ there exist $A>0$ such that
\begin{equation}
    \label{eq:analyticHF}
    |\dalpha \phi_\iota(x)| \leq  r(x)^{\min(\eta - \alpham, 0)} A^{\alpham+1} \alpham!, \qquad \iota=1,\dots, N,
  \end{equation}
  for all $x\in \bigcup_{\fc \in \fC} B_D(\fc)$ and $\alpha \in \mathbb{N}_0^d$.
  \end{theorem}
  From Theorem \ref{theorem:HFreg} and classical results on the analyticity of
  the Hartree-Fock wavefunctions away from the singular points \cite{Lewin2004},
  we directly obtain the following estimate on the wavefunctions of
  \eqref{eq:HF-integral}. Note that \eqref{eq:V-hyp} holds for $V_C$ for any $\epsilon<1$.
  \begin{corollary}
    \label{cor:HF}
    Let $Z_i$ be such that there exist a unique solution $\Phi =
    \{\phi_\iota\}_{\iota=1}^N \in (H^1(\mathbb{R}^3))^N$ to the Hartree-Fock problem
    \eqref{eq:HF-integral}, with negative eigenvalues.
    Then, for any $\eta<1$ there exists $A>0$ such that
    \begin{equation*}
    |\dalpha \phi_\iota(x)| \leq  r(x)^{\min(\eta - \alpham, 0)} A^{\alpham+1} \alpham!, \qquad \iota=1,\dots, N,
    \end{equation*}
    for all $x\in \mathbb{R}^3$ and $\alpha \in \mathbb{N}_0^3$.
  \end{corollary}
  \begin{remark}
      The result of Corollary \ref{cor:HF} can also be obtained via the
      arguments in \cite{Flad2008} or in \cite{Fournais2009}. Nonetheless, the
      result in Theorem \ref{theorem:HFreg} allows for a more general class of singular potentials,
    and the techniques used in the proof are of independent interest, as they can be extended rather straightforwardly to
    other nonlinear, elliptic systems.
    \end{remark}

The rest of this manuscript will be devoted to the proof of  Theorem \ref{theorem:HFreg}.
  \subsection{Proof of Theorem \ref{theorem:HFreg}}
Hereafter, we suppose that the potential $V$ has only one singularity, i.e,
$\fC = \{\fc\}$, set $R\leq 1$ and place
ourselves in a ball $B_R = B_{R}(\fc)$ centered in $\fc$, with $r(x) = |x-\fc|$ in $B_R$.
    The generalization to the case where $V$ has a set of isolated singularities
  is straightforward. 

  Let us formulate the induction assumption that will be used in the sequel.
\begin{ind-assumption}
   Let $\Phi = \{\phi_\iota\}_{\iota=1}^N$, $2\leq p<\infty$, $\gamma\in \mathbb{R}$, $k\in
   \mathbb{N}$, and $\Cphi, \Aphi>0$. We say that $H_\Phi(p, \gamma, k, \Cphi, \Aphi)$ holds if
   for all $\iota=1, \dots, N$, $\phi_\iota\in H^1(B_R)\cap L^\infty(B_R)$,
   $\Cphi \geq  \|\phi_\iota\|_{L^\infty(B_R)}$, and 
  \begin{equation}
    \label{eq:inductionhyp}
   \sum_{\alpham = j}\| r^{\alpham-\gamma}\dalpha\phi_\iota \|_{L^{p}(B_{R-k \rho})} \leq 
\Cphi \Aphi^j (k\rho)^{-j} j^j
  \end{equation}
  for all $j\in \mathbb{N}$ such that $1\leq j \leq k$ and $\rho\in (0, R/(2k)]$.
\end{ind-assumption}

We introduce some lemmas where---under the induction assumption---we estimate the norms of $\phi_i$ (Lemma
\ref{lemma:nonlin1}), of products $\phi_a\phi_b$ (Lemma \ref{lemma:nonlin2}), of
$u_{ab}$ (Lemma \ref{lemma:HF-u}), of the product
$u_{ab}\phi_\iota$ (Lemma \ref{lemma:nonlin3}), and of $V\phi_\iota$ (Lemma \ref{lemma:Vu}).
\begin{lemma}[Bounds on $L^{3p}$ norms of eigenfunctions]
  \label{lemma:nonlin1}
Let $p\geq 2d/3$, $0 <\gamma -
  d/p<\min(\epsilon, 2)$. There exists $\Cinterp>0$ such that, for all $\Cphi, \Aphi\geq 1$,
  for all $k\in \mathbb{N}$, $k\geq 2$, if $H_\Phi(p, \gamma, k, \Cphi, \Aphi)$ holds,
\begin{equation}
   \sum_{\alpham = j}\| r^{\frac{2-\gamma}{3}+\alpham} \dalpha \phi_\iota \|_{L^{3 p} (B_{R-k\rho})} 
\leq (d+1)\Cinterp e^\theta \Cphi \Aphi^{j+\theta} (k\rho)^{-j-\theta} j^j (j+1)^\theta,\qquad \iota=1, \dots, N,
\end{equation}
for all $1\leq j \leq k-1$, for all $\rho\in (0, R/(2k)]$, and with $\theta = \frac{2}{3}\frac{d}{p}$.
\end{lemma}
\begin{proof}
  For any $\iota\in \{1, \dots, N\}$, denote $\phi=\phi_\iota$. First, we use
  equation \eqref{eq:thetaprod} of Lemma \ref{lemma:thetaprod} in the
    Appendix in order to go back
  to integrals in $L^p$: for any $j\in\{1, \dots, k-1\}$ and for any $\alpham =j$,
   \begin{multline*}
   \| r^{\frac{2-\gamma}{3}+\alpham} \dalpha \phi \|_{L^{3 p} (B_{R-k\rho})}  \leq
\Cinterp\|r^{\alpham-\gamma} \dalpha \phi\|^{1-\theta}_{L^p(B_{R-k\rho})}
\left\{ \vphantom{\sum_{i=1}^d} (\alpham+1)^\theta \| r^{\alpham-\gamma}\dalpha \phi\|^\theta_{L^p(B_{R-k\rho})} \right. \\ 
\left.+ \sum_{i=1}^d \| r^{\alpham+1-\gamma} \partial^{\alpha}\partial_i \phi \|^\theta_{L^p(B_{R-k\rho})}\right\}.
   \end{multline*}
   By the Cauchy-Schwarz inequality,
     \begin{multline*}
        \sum_{\alpham = j} \| r^{\alpham -\gamma}\dalpha \phi\|^{1-\theta}_{L^p(B_{R-k\rho})}
(\alpham+1)^\theta \| r^{\alpham-\gamma}\dalpha \phi\|^\theta_{L^p(B_{R-k\rho})}
         \\
\leq \left(\sum_{\alpham=j} \| r^{\alpham -\gamma}\dalpha \phi\|_{L^p(B_{R-k\rho})}  \right)^{1-\theta}
              \left(\sum_{\alpham=j}
  (\alpham+1) \| r^{\alpham-\gamma}\dalpha \phi\|_{L^p(B_{R-k\rho})}
  \right)^\theta
     \end{multline*}
     and, 
     \begin{multline*}
       \sum_{\alpham = j} \left( \| r^{\alpham -\gamma}\dalpha \phi\|^{1-\theta}_{L^p(B_{R-k\rho})}
\sum_{i=1}^d\| r^{\alpham+1-\gamma} \partial^{\alpha}\partial_i \phi \|^\theta_{L^p(B_{R-k\rho})} \right)
         \\
\leq \sum_{i=1}^d\left(\sum_{\alpham=j} \| r^{\alpham -\gamma}\dalpha \phi\|_{L^p(B_{R-k\rho})}  \right)^{1-\theta}
              \left(\sum_{\alpham=j}
\| r^{\alpham+1-\gamma} \partial^{\alpha}\partial_i \phi \|_{L^p(B_{R-k\rho})}
  \right)^\theta
     \end{multline*}
   Then, hypothesis \eqref{eq:inductionhyp} implies 
   \begin{equation*}
\left(\sum_{\alpham=j}\|r^{\alpham-\gamma} \dalpha \phi\|_{L^p(B_{R-k\rho})} \right)^{1-\theta}\leq \Cphi^{1-\theta} \Aphi^{j (1-\theta)} \rho^{-j (1-\theta)} \left( \frac{j}{k}\right)^{j (1-\theta)}
   \end{equation*}
   and 
   \begin{multline*}
( j+1)^\theta \left(\sum_{\alpham=j} \| r^{j-\gamma}\dalpha \phi\|_{L^p(B_{R-k\rho})}   \right)^{\theta}
+\sum_{i=1}^d \left( \sum_{\alpham=j} \| r^{j+1-\gamma} \partial^{\alpha}\partial_i \phi \|_{L^p(B_{R-k\rho})} \right)^{\theta}
\\
\leq \Cphi^\theta (j+1)^\theta \Aphi^{j \theta} \rho^{-j\theta} \left( \frac{j}{k}\right)^{j\theta}
+d\Cphi^\theta  \Aphi^{(j +1)\theta} \rho^{-(j+1)\theta} \left( \frac{j+1}{k}\right)^{(j+1)\theta}.
   \end{multline*}
   Therefore, multiplying the right hand sides of the two last inequalities,
   \begin{equation*}
  \sum_{\alpham=j} \| r^{\frac{2-\gamma}{3}+\alpham} \dalpha u \|_{L^{3p} (B_{R-k\rho})}  \leq
   (d+1)\Cinterp\Cphi \Aphi^{j +\theta} (k\rho)^{-j -\theta} j^{j (1-\theta)} (j+1)^{(j+1)\theta}.
   \end{equation*}
   We finally need to bound the last two terms in the multiplication above:
   \begin{equation*}
     j^{j (1-\theta)} (j+1)^{(j+1)\theta} = j^j (j+1)^\theta \left( 1+\frac{1}{j} \right)^{\theta j}\leq j^j(j+1)^\theta e^\theta.
   \end{equation*}
\end{proof}
\begin{lemma}[Bounds on norms of products of eigenfunctions]
  \label{lemma:nonlin2}
 Let $p\geq 2d/3$, $0 <\gamma -
  d/p<\min(\epsilon, 2) $ and $C_\Phi, \Aphi\geq 1$.
  Let also $\theta = \frac{2}{3}\frac{d}{p}$ and
 \begin{equation}
   \label{eq:C1}
  C_1 =  \frac{(d+1)^2}{2}\Cinterp^2 e^{2\theta+1} \Cphi^2 + 2(d+1)(4\pi)^{1/2d}\Cinterp e^\theta \Cphi^2 .
 \end{equation}
 For all $k\in \mathbb{N}$, $k\geq 2$,
if $H_\Phi(p, \gamma, k, C_\Phi, \Aphi)$ holds, then
 \begin{equation}
   \label{eq:nonlin2}
   \sum_{\alpham=j}\| r^{\frac{2}{3}( 2-\gamma )+\alpham} \dalpha (\phi_\iota\phi_\kappa) \|_{L^{3 p/2} (B_{R-k\rho})} 
\leq C_1\Aphi^{j+2\theta} \rho^{-j-2\theta} \left(\frac{j}{k}\right)^j j^{1/2},\quad \iota,\kappa=1, \dots, N,
 \end{equation}
 for all $1\leq j \leq k-1$ and $\rho \in (0, R/(2k)]$.
\end{lemma}
\begin{proof}
 Denote $\phi=\phi_\iota$ and $\psi = \phi_\kappa$. By Leibniz's rule and the Cauchy-Schwarz inequality,
 \begin{multline}
   \label{eq:leibniz1}
  \| r^{\frac{2}{3}(2-\gamma)+\alpham} \dalpha (\phi\psi) \|_{L^{3 p/2} (B_{R-k\rho})} 
   \\
   \begin{aligned}
&\leq  
\sum_{0< \beta < \alpha} \binom{\alpha}{\beta} \| r^{\frac{2-\gamma}{3}+\betam} \dbeta \phi \|_{L^{3 p} (B_{R-k\rho})} \| r^{\frac{2-\gamma}{3}+\alpham-\betam} \partial^{\alpha-\beta} \psi \|_{L^{3 p} (B_{R-k\rho})} 
\\
&\quad +  \| r^{\frac{2}{3}( 2-\gamma )+\alpham} \dalpha \phi \|_{L^{3 p/2} (B_{R-k\rho})} \|  \psi \|_{L^{\infty} (B_{R-k\rho})} \\
&\quad +  \| r^{\frac{2}{3}( 2-\gamma )+\alpham} \dalpha \psi \|_{L^{3 p/2} (B_{R-k\rho})} \|  \phi \|_{L^{\infty} (B_{R-k\rho})} 
   \end{aligned}
 \end{multline}
 Consider the sum over $0<\beta <\alpha$. By manipulation on the sums and using \eqref{eq:binom},
 \begin{align*}
&\sum_{\alpham=j}\sum_{0<\beta <\alpha} \binom{\alpha}{\beta} \| r^{\frac{2-\gamma}{3}+\betam} \dbeta \phi \|_{L^{3 p} (B_{R-k\rho})} \| r^{\frac{2-\gamma}{3}+\alpham-\betam} \partial^{\alpha-\beta} \psi \|_{L^{3 p} (B_{R-k\rho})} 
   \\ & \qquad
        = \sum_{i=1}^{j-1} \sum_{\betam=i}\sum_{\substack{\alpham=j\\ \alpha>\beta}} \binom{\alpha}{\beta}
\| r^{\frac{2-\gamma}{3}+\betam} \dbeta \phi \|_{L^{3 p} (B_{R-k\rho})} \| r^{\frac{2-\gamma}{3}+\alpham-\betam} \partial^{\alpha-\beta} \psi \|_{L^{3 p} (B_{R-k\rho})} 
   \\ & \qquad
        \leq \sum_{i=1}^{j-1}\binom{j}{i} \sum_{\betam=i}\sum_{\substack{\alpham=j\\ \alpha>\beta}} 
\| r^{\frac{2-\gamma}{3}+\betam} \dbeta \phi \|_{L^{3 p} (B_{R-k\rho})} \| r^{\frac{2-\gamma}{3}+\alpham-\betam} \partial^{\alpha-\beta} \psi \|_{L^{3 p} (B_{R-k\rho})} 
   \\ & \qquad
     = \sum_{i=1}^{j-1}\binom{j}{i} \sum_{\betam=i}\sum_{\xim = j-i} 
\| r^{\frac{2-\gamma}{3}+\betam} \dbeta \phi \|_{L^{3 p} (B_{R-k\rho})} \| r^{\frac{2-\gamma}{3}+\xim} \partial^{\xi} \psi \|_{L^{3 p} (B_{R-k\rho})} 
 \end{align*}
 Hence, using Lemma \ref{lemma:nonlin1} and Stirling's inequality on the last
 line above gives
 \begin{align*}
&\sum_{\alpham=j}\sum_{0<\beta <\alpha} \binom{\alpha}{\beta} \| r^{\frac{2-\gamma}{3}+\betam} \dbeta \phi \|_{L^{3 p} (B_{R-k\rho})} \| r^{\frac{2-\gamma}{3}+\alpham-\betam} \partial^{\alpha-\beta} \psi \|_{L^{3 p} (B_{R-k\rho})} 
\\
   &\qquad
     \leq (d+1)^2\Cinterp^2 e^{2\theta} \Cphi^2 \Aphi^{j+2\theta}(k\rho)^{-j-2\theta}\sum_{i=1}^{j-1}\binom{j}{i}  
     i^i (j-i)^{j-i}(i+1)^\theta  (j-i+1)^\theta
   \\
   &\qquad
     \leq (d+1)^2\Cinterp^2 e^{2\theta} \Cphi^2  \frac{1}{2\pi}\Aphi^{j+2\theta}(k\rho)^{-j-2\theta}e^j\sum_{i=1}^{j-1} \binom{j}{i} i!(j - i)!  (i+1)^\theta (j-i+1)^\theta\frac{1}{ \sqrt{i(j-i)}}
 \end{align*}
 Now, for any $i=1, \dots, j-1$ and since $j\leq k-1$, there holds $(i+1)^\theta
 (j-i+1)^\theta\leq k^{2\theta}$. In addition as already used in
 \cite{DallAcqua2012}, by comparing the Riemann sum with the integral,
 \begin{equation*}
   \sum_{i=1}^{j-1} \frac{1}{ \sqrt{i(j-i)}} \leq \int_0^j\frac{1}{ \sqrt{i(j-i)}}di = \pi,
 \end{equation*}
 hence
\begin{multline*}
      (d+1)^2\Cinterp^2 e^{2\theta} \Cphi^2  \frac{1}{2\pi}\Aphi^{j+2\theta}(k\rho)^{-j-2\theta}e^j\sum_{i=1}^{j-1} \binom{j}{i} i!(j - i)!  (i+1)^\theta (j-i+1)^\theta\frac{1}{ \sqrt{i(j-i)}}
      \\
      \leq \frac{(d+1)^2}{2}\Cinterp^2 e^{2\theta} \Cphi^2  \Aphi^{j+2\theta}\rho^{-j-2\theta}k^{-j}e^jj!.
\end{multline*}
Using again Stirling's inequality,
\begin{multline*}
\sum_{\alpham=j}\sum_{0<\beta <\alpha} \binom{\alpha}{\beta} \| r^{\frac{2-\gamma}{3}+\betam} \dbeta \phi \|_{L^{3 p} (B_{R-k\rho})} \| r^{\frac{2-\gamma}{3}+\alpham-\betam} \partial^{\alpha-\beta} \psi \|_{L^{3 p} (B_{R-k\rho})} 
\\ \leq
   \frac{(d+1)^2}{2}\Cinterp^2 e^{2\theta+1} \Cphi^2  \Aphi^{j+2\theta}\rho^{-j-2\theta}k^{-j}j^j\sqrt{j}.
\end{multline*}
The two remaining terms at the right hand side of \eqref{eq:leibniz1} are controlled
using Lemma \ref{lemma:nonlin1} and the
boundedness of the functions in $\Phi$. Indeed
\begin{align*}
  \| r^{\frac{2}{3}( 2-\gamma )+\alpham} \dalpha \phi \|_{L^{3 p/2} (B_{R-k\rho})}
  &\leq \|r^{\frac{2-\gamma}{3} }\|_{L^{3p}(B_{R-k\rho})} \|r^{\frac{2-\gamma}{3}+\alpham} \dalpha \phi \|_{L^{3 p} (B_{R-k\rho})}  
    \\
  &\leq (4\pi)^{1/3p}R^{\frac{2-(\gamma-d/p)}{3}}\|r^{\frac{2-\gamma}{3}+\alpham} \dalpha \phi \|_{L^{3 p} (B_{R-k\rho})}  
    \\
  &\leq (4\pi)^{1/3p}(d+1)\Cinterp e^\theta \Cphi \Aphi^{j+\theta} (k\rho)^{-j-\theta} j^j (j+1)^\theta
\end{align*}
where we have used Lemma \ref{lemma:nonlin1}, the fact that $\gamma-d/p<2$, and
$R\leq 1$. Then, since $\|  \psi \|_{L^{\infty} (B_{R-k\rho})}\leq \Cphi$ by
hypothesis and $p\geq 2d/3$
\begin{equation*}
   \| r^{\frac{2}{3}( 2-\gamma )+\alpham} \dalpha \phi \|_{L^{3 p/2} (B_{R-k\rho})} \|  \psi \|_{L^{\infty} (B_{R-k\rho})} 
  \leq (4\pi)^{1/2d}(d+1)\Cinterp e^\theta \Cphi^2 \Aphi^{j+\theta} (k\rho)^{-j-\theta} j^j (j+1)^\theta.
\end{equation*}
The same holds for the last term of \eqref{eq:leibniz1}, thus concluding the proof.
\end{proof}
\begin{lemma}[Bounds on norms of the potentials $u_{ab}$]
  \label{lemma:HF-u}
  Let $\Phi = \{\phi_1, \dots, \phi_N\}$ and let $u_{ab}$, $a,b = 1, \dots, N$ be
  the solution in $\mathbb{R}^d$, $d=2,3$, to 
  \begin{equation}
    \label{eq:laplaceHF}
    -\Delta u_{ab} = 4\pi \phi_a \phi_b.
  \end{equation}
Let also $p\geq 2d/3$, $0 <\gamma -
 d/p<\min(\epsilon, 2) $, and $\Cphi , \Aphi\geq 1$ such
 that
 \begin{equation}
   \label{eq:hypCA-lemma}
  \Cphi \geq \max_{a,b=1, \dots, N}\|u_{ab}\|_{L^\infty(\mathbb{R}^d)} , \qquad \Aphi \geq 4\pi \Cregq \rationum.
 \end{equation}
     There exists $\CCp>0$ independent of $\Aphi$ such that, for all $k\in
     \mathbb{N}$, $k\geq 2$, if
 $H_\Phi(p, \gamma, k, C_\Phi, \Aphi)$ holds, then
  \begin{equation}
    \label{eq:HF-u}
    \sum_{\alpham = j} \|r^{j-\tgamma} \partial^\alpha u_{ab} \|_{L^{3p/2}(B_{R-k\rho})}
    \leq \CCp \Aphi^{j+2\theta} \rho^{-j-2\theta} \left(\frac{j}{k}\right)^j 
  \end{equation}
  for all integers $1\leq j \leq k$ and all $\rho \in (0, R/(2k)]$, and where $\tgamma = \frac{2}{3}( \gamma - 2)$ and $\theta=\frac{2}{3}\frac{d}{p}$.
\end{lemma}
\begin{proof}
  Suppose $j\geq 3$. We start by considering $j+1$ concentric balls 
  \begin{equation*}
    \tB_i= B_{R-k\frac{j-i}{j}\rho},\quad i = 0,\dots, j,
  \end{equation*}
  see Figure \ref{fig:uab-balls}.
  \begin{figure}
    \centering
    \includegraphics[width=.4\textwidth]{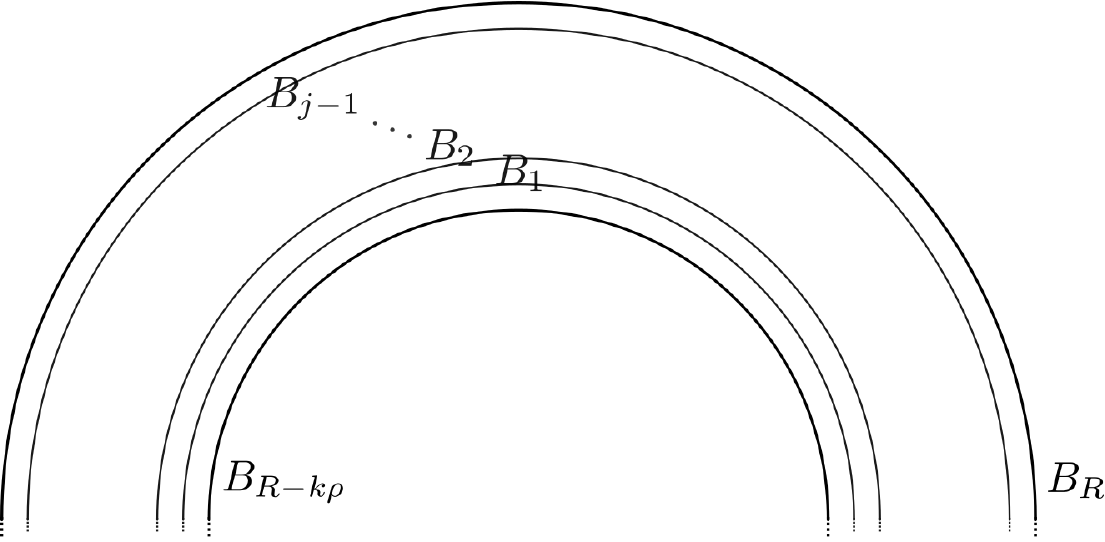}
    \caption{Concentric balls $B_i$.}
    \label{fig:uab-balls}
  \end{figure}
  Clearly, $B_{R-k\rho}  = \tB_{0} \subset \tB_{1} \subset \cdots \subset \tB_{j} = B_R$.
Now, for all $i=0, \dots, j-2$, using Proposition \ref{lemma:elliptic-weighted} in the Appendix (with $k$ replaced by $j-i-1$, $j$ replaced by $j-i-1$, $\rho$
    replaced by $\frac{k}{j}\rho$ and $\gamma$ replaced by $\tgamma$) and equation \eqref{eq:laplaceHF} we find
\begin{equation}
  \label{eq:reg-u-proof}
  \begin{aligned}
    &\sum_{\alpham = {j-i}}\| r^{-\tgamma+\alpham} \dalpha u_{ab}\|_{L^{3p/2}(\tB_i)}
      \leq \Cregq\left( 4\pi\sum_{\alpham=j-i-2}\|r^{-\tgamma + 2 +\alpham}\dalpha(\phi_a \phi_b)\|_{L^{3p/2}(\tB_{i+1})}
    \right. \\
&\qquad \left. +  \left(\rho\frac{k}{j}\right)^{-1}\sum_{\alpham=j-i-1}\|r^{-\tgamma+\alpham}\dalpha u_{ab}\|_{L^{3p/2}(\tB_{i+1})}
+  \left(\rho\frac{k}{j}\right)^{-2}\sum_{\alpham=j-i-2}\|r^{-\tgamma  +\alpham}\dalpha u_{ab}\|_{L^{3p/2}(\tB_i+1)}\right).
  \end{aligned}
\end{equation}
  We also write
  \begin{equation*}
    s_i = \Cregq^i\left( \frac{k}{j}\rho \right)^{-i}\sum_{\alpham=j-i}  \|r^{-\tgamma + \alpham} \dalpha u_{ab} \|_{L^{3p/2}(\tB_i)}\qquad i=0, \dots, j
  \end{equation*}
  and
  \begin{equation*}
    t_i = (4\pi\Cregq)^{i+1}\left( \frac{k}{j}\rho \right)^{-i}\sum_{\alpham=j-i-2}  \|r^{-\tgamma + 2+\alpham} \dalpha (\phi_a\phi_b) \|_{L^{3p/2}(\tB_{i+1})}\qquad i=0, \dots, j-2.
  \end{equation*}
  Then, since $\Cregq\geq1$ and $\tB_{i+1}\subset \tB_{i+2}$, equation \eqref{eq:reg-u-proof} implies
  \begin{equation*}
    s_i \leq t_i + s_{i+1} + s_{i+2}.
  \end{equation*}
  Let now $F_i$ denote the $i$th Fibonacci number (with $F_0 = F_1=1$). Iterating on the above, one obtains
  \begin{equation*}
    s_0 \leq \sum_{i=0}^{j-2} F_i t_i + F_{j-1} s_{j-1} + F_{j-2} s_j.
  \end{equation*}
  Denoting $\ratio = \rationum$ and remarking that $ F_i \leq \ratio^i$,
  \begin{equation}
    \label{eq:iterate-uab}
    \begin{aligned}
      \sum_{\alpham = j}\| r^{-\tgamma+\alpham} \dalpha u_{ab}\|_{L^{3p/2}(\tB_0)} 
      &\leq
 \sum_{i=0}^{j-2}  ( 4\pi \Cregq\ratio)^{i+1}\left(  \frac{k}{j}\rho\right)^{-i}\sum_{\alpham = j-i-2}\|r^{-\tgamma+2 +\alpham}\dalpha(\phi_a \phi_b)\|_{L^{3p/2}(\tB_{i+1})} \\
&\qquad  + \sum_{\alpham=0,1}(\Cregq\ratio)^{j-\alpham} \left(\frac{k}{j}\rho\right)^{-j+\alpham}\|r^{-\tgamma+\alpham}\dalpha u_{ab}\|_{L^{3p/2}(\tB_{j-\alpham})}.
    \end{aligned}
  \end{equation}
   We consider the first term at the right hand side of the above equation:
  $\phi_a$ and $\phi_b$ satisfy the hypotheses of Lemma \ref{lemma:nonlin2}
  (with $\rhotilde = \frac{j-i-1}{j}\rho$), thus, when $0\leq i\leq j-3$
  \begin{multline*}
\sum_{\alpham = j-i-2}\|r^{-\tgamma+2 +\alpham}\dalpha(\phi_a \phi_b)\|_{L^{3p/2}(\tB_{i+1})}\\
\begin{aligned}
&\leq \|r^2\|_{L^\infty(B_R)} \sum_{\betam = j-i-2}\|r ^{\betam - \tgamma}\dbeta(\phi_a \phi_b)\|_{L^{3p/2}(\tB_{i+1})} \\
&\leq 4\pi C_1  \Aphi^{j-i-2+2\theta} k^{-j+i+2}\left( \frac{j-i-1}{j}\rho\right)^{-j+i+2-2\theta} ({j-i-2})^{j-i-2} (j-i-2)^{1/2}.
\end{aligned}
  \end{multline*}
  When $i=j-2$ in the sum above, instead, we have the term
  \begin{equation}
      \label{eq:ij-2}
      \begin{aligned}
    \|r^{-\tgamma+2}\phi_a \phi_b\|_{L^{3p/2}(\tB_{i+1})} &\leq \|\phi_a\phi_b\|_{L^\infty(B_R)} \|r^{-\tgamma+2}\|_{L^{3p/2}(\tB_{i+1})}\leq R^{-\tgamma+2+\theta}4\pi \Cphi^2\leq R^24\pi\Cphi^2\\ &\leq 4\pi\Cphi^2,
      \end{aligned}
  \end{equation}
  where we have used that $\tgamma \leq  \theta$ and $R<1$.
  Hence, since $\Aphi \geq 4\pi\Cregq\ratio$ and indicating by $\zeta(\cdot)$ the
  Riemann zeta function,
  \begin{multline}
    \label{eq:sum-lemma}
    \sum_{i=0}^{j-3} (4\pi\Cregq\ratio)^{i+1}\left( \frac{k}{j}\rho\right)^{-i}\sum_{\alpham =
      j-i-2}\|r^{-\tgamma+2 +\alpham}\dalpha(\phi_a \phi_b)\|_{L^{3p/2}(\tB_{i+1})} \\
    \begin{aligned}
    &\leq
 4\pi C_1 \Aphi^{j-1+2\theta} \rho^{-j+2-2\theta} k^{-j+2}j^{j} \sum_{i=0}^{j-3}  j^{2\theta-2}
  \left(\frac{j-i-2}{j-i-1}\right)^{j-i-2}(j-i-1)^{-2\theta} (j-i-2)^{1/2}\\
    &\leq
 4\pi C_1 \Aphi^{j-1+2\theta} \rho^{-j+2-2\theta} k^{-j+2}j^{j} \sum_{i=0}^{j-3} 
  \left(\frac{j-i-2}{j-i-1}\right)^{j-i}\left( \frac{j-i-1}{j} \right)^{2-2\theta} (j-i-2)^{-3/2}\\
    &\leq \pi C_1\zeta(3/2)
 \Aphi^{j+2\theta} \rho^{-j-2\theta} k^{-j}j^{j}, 
    \end{aligned}
  \end{multline}
  where we have also used the facts that $k\rho\leq \frac{1}{2}$, and $\theta\leq 1$.

  We still need to bound the second term at the right hand side of
  \eqref{eq:iterate-uab}. There holds
  \begin{equation}
    \label{eq:uab0}
    \|r^{-\tgamma} u_{ab}\|_{L^{3p/2}(\tB_j)} \leq\|r^{-\tgamma} \|_{L^{3p/2}(B_R)} \| u_{ab}\|_{L^\infty(B_R)} \leq 4\pi \Cphi,
  \end{equation}
  by hypothesis \eqref{eq:hypCA-lemma}. Furthermore, note that by the hypotheses
  on $\gamma$ and $p$, we have $1-\tgamma\geq 0$. By classical elliptic regularity
  in Sobolev spaces \cite[Corollary D.4]{DallAcqua2012},
  there exists a constant $\CSp$ dependent only on $p$ such that
  \begin{equation}
    \label{eq:Cs}
    \begin{aligned}
      \sum_{\alpha\leq 2} \|\dalpha u_{ab}\|_{L^{3p/2}(\tB_{j-1})}
      &\leq \CSq\left( \|\phi_a\phi_b\|_{L^{3p/2}(B_{R+1})} + \| u_{ab}\|_{L^{3p/2}(B_{R+1})} \right) 
      \\
      &\leq |B_{R+1}|^{2/(3p)}\CSq\left( \|\phi_a\phi_b\|_{L^\infty(\mathbb{R}^d)} + \| u_{ab}\|_{L^\infty(\mathbb{R}^d)} \right)
      \\
      &\leq (4\pi)^{2/(3p)} (R+1)^\theta\CSq\left( \|\phi_a\phi_b\|_{L^\infty(\mathbb{R}^d)} + \| u_{ab}\|_{L^\infty(\mathbb{R}^d)} \right)
      \\ & \leq 16\pi\CSq\Cphi^2
    \end{aligned}
  \end{equation}
where we have also used $2/(3p)\leq 1$, $R\leq 1$, and \eqref{eq:hypCA-lemma}. Hence, 
  \begin{equation}
    \label{eq:uab1}
      \sum_{\alpham=1}\|r^{1-\tgamma} \dalpha u_{ab}\|_{L^{3p/2}(\tB_{j-1})}
      \leq\|r^{1-\tgamma} \|_{L^{\infty}(B_R)} \sum_{\alpham=1}\|\dalpha u_{ab}\|_{L^{3p/2}(\tB_{j-1})} 
      \leq 16\pi\CSq\Cphi^2.
  \end{equation}
  Now, combining \eqref{eq:iterate-uab}, \eqref{eq:ij-2}, \eqref{eq:sum-lemma}, \eqref{eq:uab0},
  and \eqref{eq:uab1}, we obtain
  \begin{multline*}
      \sum_{\alpham = j}\| r^{-\tgamma+\alpham} \dalpha u_{ab}\|_{L^{3p/2}(B_{R-k\rho})} 
      \\
    \leq (4\pi\Cphi^2+\pi C_1\zeta(3/2) + 16\pi\CSq\Cphi^2 + 4\pi\Cphi)
 \Aphi^{j+2\theta} \rho^{-j-2\theta} k^{-j}j^{j}
  \end{multline*}
  when $j\geq 3$. The cases $j=0,1,2$ are easily treated using \eqref{eq:uab0} and \eqref{eq:Cs}.
  \end{proof}

\begin{lemma}[Bounds on products of eigenfunctions and electronic potentials]
  \label{lemma:nonlin3}
  Let $\Phi = \{\phi_1, \dots, \phi_N\}$ and let $u_{ab}$, $a,b = 1, \dots, N$ be
  solution to \eqref{eq:laplaceHF}.
Let furthermore $p\geq 2d/3$, $0 <\gamma - d/p<\min(\epsilon, 2) $, and $\Cphi , \Aphi\geq 1$ such
 that
 \begin{equation}
   \label{eq:hypCA-lemma-2}
  \Cphi \geq \max_{a,b=1, \dots, N}\|u_{ab}\|_{L^\infty(\mathbb{R}^d)} , \qquad \Aphi \geq 4\pi \Cregq \rationum.
 \end{equation}
There exists $\CCCp$ independent of $\Aphi$ such that, for all $k\in \mathbb{N}$, if
 $H_\Phi(p, \gamma, k, C_\Phi, \Aphi)$ holds , then
 \begin{equation}
   \label{eq:nonlin3}
 \sum_{\alpham=j}  \| r^{2-\gamma+j} \dalpha (u_{ab}\phi_\iota) \|_{L^{p} (B_{R-k\rho})} 
\leq \CCCp A_{\Phi}^{j+3\theta} \rho^{-j-3\theta} \left(\frac{j}{k}\right)^j j,\qquad a,b,\iota=1, \dots, N,
 \end{equation}
  for all integer $1\leq j \leq k$, all $\rho\in (0, R/(2k)]$, and where $\theta = \frac{2}{3}\frac{d}{p}$.
\end{lemma}
\begin{proof}
Denote $u = u_{ab}$, $\phi = \phi_j$. We have  
 \begin{multline}
   \label{eq:nonlin3-1}
 \sum_{\alpham=j}  \| r^{{2-\gamma}+\alpham} \dalpha (u\phi) \|_{L^{p} (B_{R-k\rho})} 
   \\ \leq  \sum_{\alpham=j}
\sum_{\beta \leq \alpha} \binom{\alpha}{\beta} \| r^{\frac{2-\gamma}{3}+\betam} \dbeta \phi \|_{L^{3 p} (B_{R-k\rho})} \| r^{\frac{2}{3}(2-\gamma)+\alpham-\betam} \partial^{\alpha-\beta} u \|_{L^{3p/2} (B_{R-k\rho})} .
 \end{multline}
 Using \eqref{eq:nonlin2} we follow the same procedure as in the proof of Lemma
 \ref{lemma:nonlin2}. When $0<\beta<\alpha$ in the sum above, using Lemmas
 \ref{lemma:nonlin1} and \ref{lemma:HF-u},
\begin{align*}
&\sum_{\alpham=j}\sum_{0<\beta <\alpha} \binom{\alpha}{\beta} \| r^{\frac{2-\gamma}{3}+\betam} \dbeta \phi \|_{L^{3 p} (B_{R-k\rho})} \| r^{\frac{2}{3}( 2-\gamma )+\alpham-\betam} \partial^{\alpha-\beta} u \|_{L^{3 p/2} (B_{R-k\rho})}  
      \\ & \qquad
     \leq \sum_{i=1}^{j-1}\binom{j}{i} \sum_{\betam=i}\sum_{\xim = j-i} 
\| r^{\frac{2-\gamma}{3}+\betam} \dbeta \phi \|_{L^{3 p} (B_{R-k\rho})} \| r^{\frac{2}{3}(2-\gamma)+\xim} \partial^{\xi} u \|_{L^{3 p} (B_{R-k\rho})} \\
   &\qquad
     \leq (d+1)\Cinterp e^{\theta} \Cphi \CCp\Aphi^{j+3\theta}\rho^{-j-3\theta} k^{-j-\theta}\sum_{i=1}^{j-1}\binom{j}{i}  
     i^i (j-i)^{j-i}(i+1)^\theta 
   \\
   &\qquad
     \leq (d+1)\Cinterp e^{\theta} \Cphi \CCp\frac{1}{2\pi}\Aphi^{j+3\theta}\rho^{-j-3\theta} k^{-j-\theta}e^j\sum_{i=1}^{j-1} \binom{j}{i} i!(j - i)!  (i+1)^\theta \frac{1}{ \sqrt{i(j-i)}}.
   \\
   &\qquad
     \leq \frac{d+1}{2}\Cinterp e^{\theta} \Cphi \CCp\Aphi^{j+3\theta}\rho^{-j-3\theta} k^{-j}e^jj!.
   \\
   &\qquad
     \leq \frac{d+1}{2}\Cinterp e^{\theta+1} \Cphi \CCp\Aphi^{j+3\theta}\rho^{-j-3\theta} k^{-j}j^{j+1/2},
 \end{align*}
 where the last inequalities stem from the same arguments as in the proof of
 Lemma \ref{lemma:nonlin2}.
 The terms in the sum in \eqref{eq:nonlin3-1} where $\beta = 0$ and
 $\beta=\alpha$ give a similar bound: firstly, by $H_\Phi$ and Lemma \ref{lemma:HF-u}
 \begin{multline*}
   \sum_{\alpham=j}\| r^{\frac{2-\gamma}{3}}  \phi \|_{L^{3 p} (B_{R-k\rho})} \| r^{\frac{2}{3}(2-\gamma)+\alpham} \partial^{\alpha} u \|_{L^{3p/2} (B_{R-k\rho})}\\
   \begin{aligned}
  & \leq\| r^{\frac{2-\gamma}{3}}  \|_{L^{3 p} (B_{R})}\|\phi \|_{L^{\infty} (B_{R-k\rho})} \sum_{\alpham=j}\| r^{\frac{2}{3}(2-\gamma)+\alpham} \partial^{\alpha} u \|_{L^{3p/2} (B_{R-k\rho})} 
    \\ & \leq 4\pi \Cphi \CCp\Aphi^{j+2\theta} \rho^{-j-2\theta} j^{j} k^{-j}.
   \end{aligned}
 \end{multline*}
 In addition, by Lemma \ref{lemma:nonlin1} and since $\Cphi \geq \| u\|_{L^\infty(\mathbb{R}^d)}$
\begin{multline*}
   \sum_{\alpham=j}\| r^{\frac{2-\gamma}{3}+\alpham}  \dalpha \phi \|_{L^{3 p} (B_{R-k\rho})} \| r^{\frac{2}{3}(2-\gamma)}  u \|_{L^{3p/2} (B_{R-k\rho})}\\
   \begin{aligned}
     &
     \leq
(d+1)\Cinterp e^\theta \Cphi \Aphi^{j+\theta} (k\rho)^{-j-\theta} j^j (j+1)^\theta
     \| r^{\frac{2}{3}(2-\gamma)}\|_{L^{3p/2} (B_{R-k\rho})}\| u \|_{L^{\infty} (B_{R-k\rho})} 
    \\ &\leq  
(d+1)4\pi \Cinterp \Cphi e^\theta \Cphi \Aphi^{j+\theta} \rho^{-j-\theta} k^{-j}j^j
   \end{aligned}
 \end{multline*}
and choosing
\begin{equation*}
  \CCCp = \frac{d+1}{2}\Cinterp e^{\theta+1} \Cphi \CCp+4\pi \Cphi \CCp + (d+1)4\pi \Cinterp \Cphi e^\theta \Cphi 
\end{equation*}
concludes the proof.
\end{proof}

\begin{lemma}[Bounds on products of singular potential and eigenfunction]
  \label{lemma:Vu}
 Let $\Phi = \{\phi_1, \dots, \phi_N\}$ and let $V:\mathbb{R}^d\to \mathbb{R}$
 such that \eqref{eq:V-hyp} holds. Let then
 $p\geq 2d/3$, $0 <\gamma - d/p<\min(\epsilon, 2) $, and $\Cphi , \Aphi\geq 1$ such that
 \begin{equation}
   \label{eq:hypCA-lemma-Vu}
   \Aphi \geq A_V
 \end{equation}
 For all $k\in \mathbb{N}$, if
 $H_\Phi(p, \gamma, k, \Cphi, \Aphi)$ holds, then
 \begin{equation}
   \label{eq:Vu}
 \sum_{\alpham=k-1}  \| r^{2-\gamma+\alpham} \dalpha (V\phi_\iota) \|_{L^{p} (B_{R-k\rho})} 
\leq  C_4 \Aphi^{k-1} \rho^{-k+1} k^{-k+1} (k-1)^{k} ,\qquad \iota=1, \dots, N,
 \end{equation}
  for all $\rho\in (0, R/(2k)]$, with $C_4= \left(\frac{1}{2\sqrt{2\pi}}e  + 4\pi e+1\right)C_V\Cphi$.
\end{lemma}
\begin{proof}
  There holds
\begin{multline}
  \label{eq:Vu-proof-1}
  \sum_{\alpham = k-1}\| r^{2-\gamma+\alpham}\dalpha (V\phi_\iota) \|_{L^p(B_{R-k\rho})} \\ 
  \begin{aligned}
&\leq \sum_{\alpham = k-1}\sum_{0 <\beta < \alpha} \binom{\alpha}{\beta} \| r^{2-\epsilon+\betam} \dbeta V \|_{L^{\infty} (B_{R-k\rho})} \| r^{\epsilon-\gamma+\alpham-\betam} \partial^{\alpha-\beta} \phi_\iota \|_{L^{p} (B_{R-k\rho})} \\
&\quad+ \| r^{2-\epsilon} V \|_{L^{\infty} (B_{R-k\rho})}  \sum_{\alpham = k-1}\| r^{\epsilon-\gamma+\alpham} \dalpha \phi_\iota \|_{L^{p} (B_{R-k\rho})} \\
&\quad+\sum_{\alpham = k-1}\| r^{2-\epsilon+\alpham} \dalpha V \|_{L^{\infty} (B_{R-k\rho})} \| r^{\epsilon-\gamma}  \phi_i \|_{L^{p} (B_{R-k\rho})} 
  \end{aligned}
\end{multline}
By the usual manipulations,
\begin{multline*}
\sum_{\alpham = k-1}\sum_{0 <\beta < \alpha} \binom{\alpha}{\beta} \| r^{2-\epsilon+\betam} \dbeta V \|_{L^{\infty} (B_{R-k\rho})} \| r^{\epsilon-\gamma+\alpham-\betam} \partial^{\alpha-\beta} \phi_\iota \|_{L^{p} (B_{R-k\rho})} \\
\begin{aligned}
&\leq C_V \Cphi \sum_{j=1}^{k-2} \binom{k-1}{j} A_V^{j} \Aphi^{k-1-j} j!(k-1-j)^{k-1-j}(k\rho)^{-k+1+j}\\
&\leq \frac{1}{\sqrt{2\pi}}C_V \Cphi \Aphi^{k-1}(k-1)!e^{k-1}\sum_{j=1}^{k-2}(k-1-j)^{-1/2}(k\rho)^{-k+1+j}\\
&\leq \frac{1}{\sqrt{2\pi}}C_V \Cphi \Aphi^{k-1}(k-1)!e^{k-1}(k\rho)^{-k+2}\\
&\leq \frac{1}{2\sqrt{2\pi}}e C_V \Cphi \Aphi^{k-1}(k-1)^{k-1/2}(k\rho)^{-k+1}
\end{aligned}
\end{multline*}
The bound on the second to last term in \eqref{eq:Vu-proof-1} is straightforward, while
for the last term we note that $\epsilon-\gamma> -d/p$ thus
$\|r^{\epsilon-\gamma}
\phi_\iota\|_{L^p(B_R)} \leq 4\pi \Cphi$ and
\begin{align*}
  \sum_{\alpham = k-1}\| r^{2-\epsilon+\alpham} \dalpha V \|_{L^{\infty} (B_{R-k\rho})} \| r^{\epsilon-\gamma}  \phi_i \|_{L^{p} (B_{R-k\rho})} 
  &\leq 4\pi C_\phi C_VA_V^{k-1} (k-1)!\\
  &\leq 4\pi C_\phi C_V e A_V^{k-1} (k-1)^{k-1/2} e^{-k+1}.
\end{align*}
Therefore,
\begin{equation*}
  \sum_{\alpham = k-1}\| r^{2-\gamma+\alpham}\dalpha (V\phi_\iota) \|_{L^p(B_{R-k\rho})} \
  \leq \left(\frac{1}{2\sqrt{2\pi}}e  + 4\pi e+1\right)C_V\Cphi \Aphi^{k-1}(k-1)^k (k\rho)^{-k+1}
\end{equation*}
and this concludes the proof.
\end{proof}
  \begin{proof}[Proof of Theorem \ref{theorem:HFreg}]
     First, we remark that $\phi_a, \phi_b \in H^1(\mathbb{R}^d)$ implies
   $u_{ab}\in W^{2,3}(B_R)$ via the second equation of \eqref{eq:HF}.
   Due to \eqref{eq:V-hyp}, there exists $q> d/2$ such that $V\in L^q(B_R)$,
   and, by classical elliptic regularity arguments
   \cite{Stampacchia1965}, $\Phi\in \left( L^\infty(B_R) \right)^N$.
   Hence, for all $a,b\in \{1, \dots, N\}$, $\phi_a\phi_b \in H^1(B_R)\cap
     L^\infty(B_R)$. Therefore, by \eqref{eq:HF} again, $u_{ab}\in
     H^3(B_R)\subset W^{1, \infty}(B_R)$ for all $1\leq p<\infty$. This implies
     that $\Phi \in \left(\cJ^2_\xi(B_R)  \right)^N$, for all $\xi-d/2<\epsilon$. We can
     conclude that, for all $a, b\in \{1, \dots, N\}$ and all $\iota\in \{1,
     \dots, \infty\}$, there holds $u_{ab}\phi_\iota \in \cJ^2_{\gamma}(B_R)$,
     which in turn implies $\Phi \in \left( \cJ^4_\xi \right)^N$, for all
     $\xi -d/2 < \epsilon$. This implies furthermore, by \cite[Lemma 3.1]{Maday2019a}, 
   \[\sum_{\alpham = 2}\| r^{2-\gamma}\dalpha\phi_\iota\|_{L^p(B_R)} < \infty\] for all $p>1$, $\gamma < d/p+\epsilon$, and $\iota=1,\dots,N$.
Hence, for all
$1<p<\infty$ and $\gamma - d/p < \epsilon$, there exist $C, A>0$ (dependent on
$p$ and $\gamma$) such that $H_{\Phi}(p, \gamma, 2, C, A)$ holds.

\textbf{Induction step.}
We proceed by induction and impose a restriction on $p$; specifically, we fix a
finite $\phat$ such that
\begin{equation}
  \label{eq:p-cond}
\phat \geq 2d .
\end{equation}
We denote the corresponding $\thetahat = \frac{2}{3}\frac{d}{\phat}$.
Let us now also fix $\ghat\in \mathbb{R}$ such that $0 < \ghat -d/\phat<\min(\epsilon, 2)$. Let then $\Cphi, \Aphi\geq 1$ such that $H_{\Phi}(\phat, \ghat, 2,
\Cphi, \Aphi)$ holds, and that
\begin{equation}
  \label{eq:CA-proof}
\begin{aligned}
  &\Cphi \geq \max_{a, b =1,\dots, N} \| u_{ab} \|_{L^\infty(\mathbb{R}^d)} \\
  &\Aphi \geq \max\left( A_V , 4\pi\Cregqstar\rationum , \Cregpstar\left(C_4 + N^3\max_{a,b,\sigma, \iota}|c_{ab}^{\iota\sigma}|\CCCpstar + (N\max_{\iota}|\lambda_{\iota}| +2)\Cphi   \right) \right).
\end{aligned}
\end{equation}
Note that such constants fulfill the hypotheses of Lemmas \ref{lemma:nonlin1} to \ref{lemma:Vu}.
Suppose now that the induction hypothesis $H_\Phi(\phat, \ghat, k, \Cphi, \Aphi)$
holds for a $k\in \mathbb{N}$, $k\geq 2$: we will show that $H_\Phi(\phat, \ghat,
k+1, \Cphi, \Aphi)$ holds.

We start by remarking that, for all $\rho \in (0, R/(2(k+1))]$, there exists
$\rhotilde = \frac{k+1}{k}\rho$, so that, by induction hypothesis, for all $j=1,
\dots, k$ and all $\iota = 1, \dots, N$.
\begin{align*}
  \sum_{\alpham = j}\| r^{\alpham-\gamma}\dalpha\phi_\iota \|_{L^{p}(B_{R-(k+1) \rho})}
 & = 
   \sum_{\alpham = j}\| r^{\alpham-\gamma}\dalpha\phi_\iota \|_{L^{p}(B_{R-k \rhotilde})}  
   \leq
   \Cphi \Aphi^j (k\rhotilde)^{-j} j^j
   \\ & = 
   \Cphi \Aphi^j ((k+1)\rho)^{-j} j^j.
\end{align*}
We still have to show that
\begin{equation}
  \label{eq:toshow}
  \sum_{\alpham = k+1}\| r^{\alpham-\gamma}\dalpha\phi_\iota \|_{L^{p}(B_{R-(k+1) \rho})} \leq C_\Phi A_\Phi^{k+1}\rho^{-(k+1)}, \qquad \iota=1, \dots, N.
\end{equation}
From \eqref{eq:HF} and \eqref{eq:elliptic-weighted}, for all $\iota = 1, \dots, N$,
\begin{multline}
  \label{eq:analytic-proof1}
    \sum_{|\alpha| = k+1} \| r^{k+1-\gamma} \dalpha
    \phi_\iota\|_{L^\phat(B_{R-(k+1)\rho})} \\
    \begin{aligned}
    &\leq \Cregpstar \left( \sum_{\alpham= k-1} \| r^{k+1-\gamma}\dalpha\left(V\phi_\iota +
    \sum_{\sigma =1}^N \sum_{a<b}c^{\iota\sigma}_{ab} u_{ab} \phi_\sigma -  \lambda_{\iota} \phi_\iota\right) \|_{L^\phat(B_{R-k\rho})}\right.
    \\
    & \left.\quad+ \sum_{|\alpha| = k-1, k}\rho^{|\alpha|-k-1}\|r^{|\alpha|-\gamma}\dalpha \phi_\iota\|_{L^\phat(B_{R-|\alpha|\rho})} \right).
    \end{aligned}
\end{multline}
Due to Lemma \ref{lemma:Vu},
\begin{equation}
  \label{eq:Vu-proof}
 \sum_{\alpham= k-1} \| r^{k+1-\gamma}\dalpha\left(V\phi_\iota\right)\|_{L^\phat(B_{R-k\rho})} \leq
 C_4 \Aphi^{k-1} \rho^{-k+1} k^{-k+1} (k-1)^{k}
  \leq
 C_4 \Aphi^{k-1} \rho^{-k},
\end{equation}
where we have used $k\leq 1/\rho$
Furthermore, from Lemma \ref{lemma:nonlin3},
\begin{multline}
  \label{eq:nonlin-proof}
\sum_{\sigma=1}^N\sum_{a<b}|c_{ab}^{\iota\sigma}|\sum_{\alpham= k-1} \| r^{{2-\gamma}+\alpham} \dalpha \left( u_{ab}\phi_\sigma \right) \|_{L^{p} (B_{R-k\rho})} 
\\
\begin{aligned}
&\leq
N^3\max_{a,b,\iota, \sigma} |c_{ab}^{\iota\sigma}|\CCCpstar \Aphi^{k-1+3\thetahat} \rho^{-k+1-3\thetahat} (k-1)^{k-1} k^{-k+1} (k-1)
\\
& 
\leq
N^3\max_{a,b,\iota, \sigma} |c_{ab}^{\iota\sigma}|\CCCpstar \Aphi^{k-1+3\thetahat} \rho^{-k+1-3\thetahat} (k-1)
\\
& 
\leq
 N^3\max_{a,b,\iota, \sigma} |c_{ab}^{\iota\sigma}|\CCCpstar \Aphi^{k-1+3\thetahat} \rho^{-k-3\thetahat} 
\\
& 
\leq
N^3\max_{a,b,\iota, \sigma} |c_{ab}^{\iota\sigma}|\CCCpstar \Aphi^{k} \rho^{-k-1},
\end{aligned}
\end{multline}
where we have used, in the last two inequalities, the facts that $k-1\leq 1/\rho$
and that $3\thetahat  = 2d/\phat \leq 1$ due to \eqref{eq:p-cond}.

Finally, from the induction hypothesis,
\begin{equation}
  \label{eq:lambda-proof}
\sum_{\alpham = k-1}\sum_{\sigma =1}^N |\lambda_{\iota,\sigma} |  \|r^{k+1-\ghat}\dalpha\phi_\sigma \|_{L^\phat(B_{R-k\rho})}
\leq N\max_{\iota}|\lambda_{\iota} |  \Cphi \Aphi^{k-1} \rho^{-k+1}
\end{equation}
and 
\begin{equation}
  \label{eq:last-proof}
  \sum_{|\alpha| = k-1, k}\rho^{|\alpha|-k-1}\|r^{|\alpha|-\gamma}\dalpha \phi_\iota\|_{L^\phat(B_{R-k\rho})} 
  \leq
  \Cphi\Aphi^{k-1} \rho^{-k-1} + \Cphi\Aphi^{k} \rho^{-k-1}.
\end{equation}
From \eqref{eq:analytic-proof1}, using the triangular inequality, and
inequalities \eqref{eq:Vu-proof}, \eqref{eq:nonlin-proof},
\eqref{eq:lambda-proof}, and \eqref{eq:last-proof}, we obtain
\begin{multline*}
  \sum_{|\alpha| = k+1} \| r^{k+1-\gamma} \dalpha
    \phi_\iota\|_{L^\phat(B_{R-(k+1)\rho})} 
    \\
    \leq \Cregpstar\left(C_4 + N^3\max_{a,b,\sigma, \iota}|c_{ab}^{\iota\sigma}|\CCCpstar + (N\max_{\iota}|\lambda_{\iota}| +2)\Cphi   \right) \Aphi^k \rho^{-k-1}.
\end{multline*}
Therefore, \eqref{eq:toshow} holds thanks to \eqref{eq:CA-proof}, i.e., 
\begin{equation*}
 H_\Phi(\phat, \ghat,
k+1, \Cphi, \Aphi) 
\end{equation*}
holds. Therefore, by induction, $H_\Phi(\phat, \ghat, k, \Cphi, \Aphi)$ holds for all $k\in \mathbb{N}$.

\textbf{Analytic estimates in the $L^\infty$ norm.}
By Lemma \ref{lemma:imbedding} and since we have shown that
\eqref{eq:inductionhyp} holds for all $k\in \mathbb{N}$,
\begin{equation*}
 \|r^{-\eta+\alpham} \dalpha \phi_\iota \|_{L^\infty(B_{R-k\rho})} \leq \Cphi \alpham^2 \Aphi^\alpham (k\rho)^{-\alpham}\alpham^\alpham,
\end{equation*}
for all $\alpham \in \mathbb{N}$ and $\rho\in (0, R/(2k)]$.
Therefore, due to Stirling's inequality and since $R-k\rho \geq R/2
$, for all $0< \eta < \epsilon$ there exist
constants $\Ctil, \Atil>0$ such that
\begin{equation}
  \label{eq:analyticHF-proofend}
 \|r^{-\eta+\alpham} \dalpha \phi_i \|_{L^\infty(B_{R/2}(\fc))} \leq \Ctil \Atil^\alpham \alpham !.
\end{equation}
\end{proof}

\appendix

\section{Technical tools in weighted spaces}
\label{appendix:technical}
The results presented in this paper rely heavily on the theory of Kondrat'ev-type
weighted Sobolev spaces, that we introduce here. We also recall---mostly from
\cite{Maday2019b}, for self-containedness---a series of technical results that are ultimately necessary for the
proof of Theorem \ref{theorem:HFreg}. 

We denote by $\Omega\subset\mathbb{R}^d$, $d=2,3$, a bounded domain with smooth boundary
and consider the case of a single singular point $\fc\in \Omega$ lying in the
interior of the domain. The generalization to the case of multiple singular
points is straightforward. We denote by $r(x) = |x-\fc|$ the distance of a point
$x\in\mathbb{R}^d$ from the singular point. In the whole appendix, we
denote by $B_R = B_R(\fc)$ $d$-dimensional balls centered in $\fc$ of radius $R>0$.
Finally, for $k\in \mathbb{N}_0$ and $1\leq p\leq \infty$, we denote by
  $W^{k, p}(\Omega)$ the classical $L^p(\Omega)$-based Sobolev spaces of order $k$.

\subsection{Weighted Sobolev spaces}
For integer $k \in\mathbb{N}_0$, a real weight exponent $\gamma\in \mathbb{R}$, 
and summability exponent $1\leq p < \infty$, 
we introduce the \emph{homogeneous weighted Sobolev spaces} 
$\cK^{k, p}_\gamma(\Omega)$. 
Given the seminorm
\begin{equation}
  \label{eq:Kseminorm}
|w|_{\cK^{k,p}_\gamma(\Omega)} 
= 
\left( \sum_{\alpham= k} \|r^{\alpham -\gamma}\dalpha w\|^p_{L^p(\Omega)}  \right)^{1/p},
\end{equation}
so that the spaces $\cK^{k, p}_\gamma(\Omega)$ are normed by
\begin{equation*}
\|w\|_{\cK^{k,p}_\gamma(\Omega)}  
= 
\left( \sum_{j=0}^{ k} |w|^p_{\cK^{j,p}_\gamma(\Omega)} \right)^{1/p}.
\end{equation*}
We denote the weighted Kondrat'ev type spaces of infinite regularity by
\begin{equation*}
\cK^{\infty,p}_\gamma(\Omega) = \bigcap_{k\in\mathbb{N}} \cK^{k, p}_\gamma(\Omega).
\end{equation*}
Furthermore, for constants $C, A>0$
we introduce the \emph{homogeneous weighted analytic-type} class
\begin{equation*}
\cK^{\varpi,p}_\gamma(\Omega; A) 
= 
\left\{ v\in \cK^{\infty,p}_\gamma(\Omega): \left| v\right|_{\cK^{k, p}_\gamma(\Omega)}
\leq 
A^{k+!}k!, \, \text{for all }k\in\mathbb{N}_0 \right\}.
\end{equation*}
Consider a continuous function $u$: we remark that, if $\gamma>d/p$, then $u\in \cK^{0,p}_{\gamma}(\Omega)$ only if
$u(\fc) = 0$. This condition is clearly not
fulfilled by solutions to \eqref{eq:HF-integral}, which are, in general, nonzero
at the singular point of the potential. For this reason, 
our focus will be mostly on \emph{non-homogeneous weighted Sobolev spaces}
The \emph{non-homogeneous analytic classes} 
are given by
\begin{equation}
\label{eq:Janalytic}
\cJ^{\varpi,p}_\gamma(\Omega; A) 
= 
\left\{ v\in W^{\lfloor \gamma-d/p \rfloor, p}(\Omega): | v |_{\cK^{k,p}_\gamma(\Omega)}
\leq 
 A^{k+1}k!, \text{for all }k\in \mathbb{N}_0: k>\gamma-d/p \right\}.
\end{equation}
For a detailed analysis of the
relationship between homogeneous and non homogeneous spaces, we refer the reader
to \cite{Kozlov1997} and \cite{Costabel2010a}.
\begin{remark}
  Using definition \eqref{eq:Janalytic}, the thesis of Theorem
  \ref{theorem:HFreg} can be restated as: for all $\eta < \epsilon$, there
  exists $A>0$ such that
  \begin{equation*}
     \phi_\iota \in \cJ^{\varpi, \infty}_{\eta}(\cup_{\fc\in\fC}B_D(\fc); A),\qquad \forall\, \iota=1, \dots, N.
  \end{equation*}
\end{remark}
\subsection{Local elliptic estimate}
\label{sec:localelliptic}
We report here, for the sake of self-containedness, a result on local weighted
elliptic regularity. This has already been introduced in \cite{Maday2019b}, and
has been proven, as an intermediate result, in \cite{Costabel2012}.
We denote the commutator by square brackets, i.e., we write
\begin{equation*}
  \left[ A, B \right] = AB-BA.
\end{equation*}
\begin{proposition}
  \label{lemma:elliptic-weighted}
  Let $1<p<\infty$, $R>0$, and $\gamma \in \mathbb{R}$. Then, there exists
    $\Cregp \geq 1$ such that for all $k\in\mathbb{N}$ and $\rho\in(0, \frac{R}{2(k+1)}]$.
   and $j\in \mathbb{N}$ such that $1\leq j \leq k$,
  \begin{multline}
    \label{eq:elliptic-weighted}
    \sum_{|\alpha| = k+1} \| r^{k+1-\gamma} \dalpha u\|_{L^p(B_{R-(j+1)\rho})}
    \leq \Cregp \left(\sum_{|\beta| = k-1} \| r^{k+1-\gamma}\dbeta(\Delta u)\|_{L^p(B_{R-j\rho})} \right.\\ \left.
+ \sum_{|\alpha| = k}\rho^{-1}\|r^{|\alpha|-\gamma}\dalpha u\|_{L^p(B_{R-j\rho})}
+ \sum_{|\alpha| = k-1}\rho^{-2}\|r^{|\alpha|-\gamma}\dalpha u\|_{L^p(B_{R-j\rho})}.
  \right)
  \end{multline}
\end{proposition}
For the proof of Proposition \ref{lemma:elliptic-weighted}, we introduce a smooth cutoff function  
$\eta\in C^\infty_0(B_{R-j\rho})$ such that for $\alpha \in \mathbb{N}^d$,
$\alpham \leq 2$
\begin{equation}
  \label{eq:etadef}
  0\leq \eta \leq 1, \qquad \eta=1 \text{ on }B_{R-(j+1)\rho},\qquad \left| \dalpha\eta\right|\leq C_\eta \rho^{-|\alpha|},
\end{equation}
and we introduce an auxiliary estimate (see \cite{Maday2019b} for the proof)
\begin{lemma}{\cite[Lemma 9]{Maday2019b}}
  \label{lemma:commutator1}
Let $1<p<\infty$, $R>0$, and $\gamma \in \mathbb{R}$ .
There exists $C>0$ such that, for all
  $\beta \in \mathbb{N}^d_0$, $\rho\in(0,
  \frac{R}{2(|\beta|+2)}]$, and $j\in \mathbb{N}$ such that $1 \leq j \leq |\beta|+1$,
  \begin{equation}
    \label{eq:commutator1}
    \sum_{|\alpha|=2}\|\left[ \dalpha, r^{|\beta|+2-\gamma}\right] \eta \dbeta u\|_{L^p(B_{R-j\rho})}\leq 
C \sum_{|\alpha| \leq 1} \rho^{-2+|\alpha|}\| r^{|\beta|+|\alpha|-\gamma} \dab u\|_{L^p(B_{R-j\rho})},
  \end{equation}
  and $C$ depends only on $\gamma$, $R$.
\end{lemma}
\begin{proof}[Proof of Proposition \ref{lemma:elliptic-weighted}]
 Let us consider a multiindex $\beta$. First,
 \begin{multline}
   \label{eq:ellipticreg1}
   \sum_{|\alpha|=2} \|r^{|\beta|+2-\gamma }\dab u\|_{L^p(B_{R-(j+1)\rho})} \leq 
\sum_{|\alpha|=2}\left\{ \|\dalpha \left(r^{|\beta|+2-\gamma} \dbeta u\right)\|_{L^p(B_{R-(j+1)\rho}) } \right. \\ 
\left.+ \|\left[\dalpha, r^{|\beta|+2-\gamma} \right]\dbeta u\|_{L^p(B_{R-(j+1)\rho}) }\right\}.
 \end{multline}
 We consider the first term at the right hand side: using \eqref{eq:etadef}
 \begin{equation*}
\sum_{|\alpha|=2} \|\dalpha \left(r^{|\beta|+2-\gamma} \dbeta u\right)\|_{L^p(B_{R-(j+1)\rho}) } \leq 
\sum_{|\alpha|=2} \|\dalpha \left(r^{|\beta|+2-\gamma} \eta\dbeta u\right)\|_{L^p(B_{R-j\rho}) }
 \end{equation*}
and by elliptic regularity and using the triangular inequality, there exists
$C_\Delta$ depending only on $p$ and $R$ such that
\begin{multline*}
\sum_{|\alpha|=2} \|\dalpha \left(r^{|\beta|+2-\gamma} \eta\dbeta u  \right)\|_{L^p(B_{R-j\rho}) }
\\
\begin{aligned}
&\leq
C_\Delta  \|\Delta \left(r^{|\beta|+2-\gamma} \eta\dbeta u\right)\|_{L^p(B_{R-j\rho}) }\\
&\leq 
C_\Delta \left(  \|r^{|\beta|+2-\gamma} \eta\Delta \dbeta u\|_{L^p(B_{R-j\rho}) } 
+ 
  \|\left[\Delta,r^{|\beta|+2-\gamma}\right] \eta\dbeta u\|_{L^p(B_{R-j\rho}) } 
+
  \|r^{|\beta|+2-\gamma}\left[\Delta,\eta\right] \dbeta u\|_{L^p(B_{R-j\rho}) } 
 \right).
\end{aligned}
\end{multline*}
Combining the last inequality with \eqref{eq:ellipticreg1} we obtain
\begin{multline}
  \label{eq:reg-p1}
   \sum_{|\alpha|=2} \|r^{|\beta|+2-\gamma }\dab u\|_{L^p(B_{R-(j+1)\rho})} \\
\begin{aligned}
 &\leq 
  C_\Delta \left(  \|r^{|\beta|+2-\gamma} \eta\dbeta\left(\Delta u\right)\|_{L^p(B_{R-j\rho}) } 
   +
    \sum_{i=1}^d\| r^{|\beta|+2-\gamma} \left(\partial_{ii}\eta\right) \dbeta u\|_{L^p(B_{R-j\rho}) } \right.\\
&\qquad\left.+ 2\sum_{i=1}^d\| r^{|\beta|+2-\gamma} \left(\partial_{i}\eta\right) \partial^{\beta}\partial_i u\|_{L^p(B_{R-j\rho}) } \right) 
+
 (1+C_\Delta)\sum_{|\alpha|=2}\|\left[\dalpha, r^{|\beta|+2-\gamma} \right]\dbeta u\|_{L^p(B_{R-j\rho}) }.
\end{aligned}
\end{multline}
The bounds on the derivatives of $\eta$ given in \eqref{eq:etadef} and the
estimate of Lemma \ref{lemma:commutator1} applied to \eqref{eq:reg-p1} then imply
the existence of a constant $C$ dependent on $p$, $\gamma$, and
$R$ such that
\begin{align*}
   \sum_{|\alpha|=2} \|r^{|\beta|+2-\gamma }\dab u\|_{L^p(B_{R-(j+1)\rho})} 
  &\leq
  C_\Delta\|r^{|\beta|+2-\gamma} \eta\dbeta\left(\Delta u\right)\|_{L^p(B_{R-j\rho}) } \\
  &\quad +
    C\sum_{|\alpha|\leq 1} \rho^{-2+|\alpha|} \|r^{|\beta|+|\alpha| -\gamma} \dab u\|_{L^p(B_{R-j\rho})}.
\end{align*}
We can now sum over all multi indices $\beta$ such that $|\beta| = k-1$ to obtain the thesis \eqref{eq:elliptic-weighted}.
\end{proof}
\subsection{Weighted interpolation estimate}
  \label{sec:thetaprod}
\begin{lemma}
  \label{lemma:thetaprod}
  Let $R>0$ such that $B_R\subset B_1$, $\gamma-d/p \geq - 2/3$, and $p\geq
  \frac{2}{3}d$. There exists a constant $\Cinterp>0$ such that for all
  $\beta\in \mathbb{N}^d_0$ and $u\in \mathcal{K}^{\betam+1,
    p}_\gamma(B_R)$ the following ``interpolation'' estimate holds
  \begin{multline}
  \label{eq:thetaprod}
  \|r^{\frac{2-\gamma}{3}+|\beta|}\dbeta u \|_{L^{3 p}(B_R)} \leq \Cinterp\|r^{\betam-\gamma} \dbeta u\|^{1-\theta}_{L^p(B_R)}
\left\{ \vphantom{\sum_{i=1}^d} (\betam+1)^\theta \| r^{\betam-\gamma}\dbeta u\|^\theta_{L^p(B_R)} \right. \\ 
\left.+ \sum_{i=1}^d \| r^{\betam+1-\gamma} \partial^{\beta}\partial_i u \|^\theta_{L^p(B_R)}\right\},
  \end{multline}
  with $\theta = \frac{2}{3}\frac{d}{p}$.
\end{lemma}
\begin{proof}
  Consider a dyadic decomposition of $B_1$ given by the sets 
  \begin{equation*}
    V^j = \left\{ x\in B_1 : 2^{-j} \leq |x|\leq 2^{-j+1} \right\}, \; j=1, 2, \dots
  \end{equation*}
   and decompose the ball $B_R$ into its intersections with the sets belonging to the decomposition, i.e., into $B^j = B_R \cap V^j$. Let us introduce the linear maps $\chi_j : V^1\to V^j$ and write with a hat the pullback of functions by $\chi_j^{-1}$, e.g, $\hr = r\circ\chi_j^{-1}$ and $\hB^j = \chi_j^{-1}(B^j)$. Then, 
   \begin{equation*}
     \|r^{\frac{2-\gamma}{3}+\betam} \dbeta u\|_{L^{3 p}(B^j)} \leq 
     2^{\frac{j}{3}(\gamma-2-d/p)}
     \|\hr^{\frac{2-\gamma}{3}+\betam} \hat{\partial}^\beta \hat{u}\|_{L^{3 p}(\hB^j)} 
   \end{equation*}
   We can now use the interpolation inequality
   \begin{equation*}
     \| v\|_{L^{3 p}(B)} \leq C \|v\|^{1-\theta}_{L^p(B)} \|v\|^\theta_{W^{1, p}(B)},
   \end{equation*}
   for $B\subset\mathbb{R}^d$, $v\in W^{1, p}(B)$ and with $\theta$ defined as above, see \cite{DallAcqua2012}. Therefore,
   \begin{equation}
     \label{eq:interp-ref}
     \|r^{\frac{2-\gamma}{3}+\betam} \dbeta u\|_{L^{3 p}(B^j)} 
      \leq C
     2^{\frac{j}{3}(\gamma-2-d/p)}
     \|\hr^{\frac{2-\gamma}{3}+\betam} \hat{\partial}^\beta \hat{u}\|^{1-\theta}_{L^{p}(\hB^j)} 
      \sum_{\alpham = 1} \|\hat{\partial}^\alpha \hr^{\frac{2-\gamma}{3}+\betam}\hat{\partial}^\beta \hat{u}\|^\theta_{L^p(\hB^j)}.
   \end{equation}
   Let us now consider the first norm in the product above. Since $\hr \in
   (1/2,1)$, we can inject in the norm a term $\hr^{\frac{2}{3}\gamma }\leq
   \max(1,2^{\frac{2}{3}|\gamma|} )= C(\gamma)$, i.e.,
     \begin{equation*}
     \|\hr^{\frac{2-\gamma}{3}+\betam} \hat{\partial}^\beta \hat{u}\|^{1-\theta}_{L^{p}(\hB^j)}  \leq C \|\hr^{\betam-\gamma} \hat{\partial}^\beta \hat{u}\|^{1-\theta}_{L^{p}(\hB^j)} .
\end{equation*}
We now compute more explicitly the second norm in the product in \eqref{eq:interp-ref}:
\begin{equation*}
      \sum_{\alpham = 1} \|\hat{\partial}^\alpha \hr^{\frac{2-\gamma}{3}+\betam}\hat{\partial}^\beta \hat{u}\|^\theta_{L^p(\hB^j) }
  \leq \left( \betam + \frac{2-\gamma}{3}\right)^\theta \|\hr^{\frac{2-\gamma}{3}+\betam-1} \hat{\partial}^\beta \hat{u}\|^\theta_{L^p(\hB^j) } 
+ \sum_{i=1}^d \|\hr^{\frac{2-\gamma}{3}+\betam} \hat{\partial}^{\beta }\partial_i\hat{u}\|^\theta_{L^p(\hB^j) }
\end{equation*}
and we may adjust the exponents of $\hr$ and the term in $\frac{2-\gamma}{3}$ introducing a constant that depends on $\gamma$, $d$ and $p$, obtaining
\begin{equation*}
      \sum_{\alpham = 1} \|\hat{\partial}^\alpha \hr^{\frac{2-\gamma}{3}+\betam}\hat{\partial}^\beta \hat{u}\|^\theta_{L^p(\hB^j) } 
\leq C
  \left( \betam + 1\right)^\theta \|\hr^{\betam -\gamma} \hat{\partial}^\beta \hat{u}\|^\theta_{L^p(\hB^j) } 
+ \sum_{i=1}^d \|\hr^{\betam-\gamma+1} \hat{\partial}^{\beta }\partial_i\hat{u}\|^\theta_{L^p(\hB^j) }.
\end{equation*}
Scaling everything back to $B^j$ and adjusting the exponents,
\begin{multline*}
  \|r^{\frac{2-\gamma}{3}+\betam} \dbeta u\|_{L^{3 p}(B^j)} \leq 
  C 2^{j\left(\gamma-d/p-2/3\right)}
  \|r^{\betam-\gamma} {\partial}^\beta u\|^{1-\theta}_{L^{p}(B^j)} 
\left\{ \vphantom{\sum_{i=1}^d}\left( \betam + 1\right)^\theta \|r^{\betam -\gamma} {\partial}^\beta {u}\|^\theta_{L^p(B^j) }  \right.\\
\left. + \sum_{i=1}^d \|r^{\betam-\gamma+1} {\partial}^{\beta }\partial_i{u}\|^\theta_{L^p(B^j) }\right\}.
\end{multline*}
If $\gamma-d/p \geq -2/3$ then we can sum over all $j=1,2, \dots$ thus obtaining the estimate \eqref{eq:elliptic-weighted} on the whole ball $B_R$.
\end{proof}

\subsection{An imbedding result}
\begin{lemma}
  \label{lemma:imbedding}
  Let $p\geq 2$, $R>0$, and $\gamma \in \mathbb{R}$ such that $\gamma -d/p > 0 $.
  Then, there exists $C$ such that
for all $\ell\in \mathbb{N}_0$ and all $v\in \cK^{\infty,p}_\gamma(B_R)$.
  \begin{equation*}
   \|v\|_{\cK^{\ell, \infty}_{\gamma-d/p}(B_R)}\leq C(\ell+1)^2\|v\|_{\cK^{\ell+2, p}_\gamma(B_R)}.
  \end{equation*}
\end{lemma}
\begin{proof}
We prove the lemma for $R=1$; the general case with $R>0$ follows by homothety (with
constants depending on $R$). Consider the annuli
  \begin{equation*}
    \Gamma_j = \left\{ x\in B_1 : 2^{-j-1} < |x| < 2^{-j} \right\},\,j \in \mathbb{N}_0
  \end{equation*}
  and let $\widehat{\Gamma}=\Gamma_0$. For all $j\in \mathbb{N}$, let $\chi_j$
  be the homothety from $\widehat{\Gamma}$ to $\Gamma_j$ and denote with a hat
  the quantities rescaled on $\widehat{\Gamma}$, e.g., $\hat{v} = v \circ \chi$. Then,  by a
   scaling argument and since $1/2 < \hat{r}_{|_{\widehat{\Gamma}}}< 1$, we have 
  \begin{equation*}
    \max_{\alpham \leq \ell}\|r^{\alpham - \gamma +d/p}\dalpha v \|_{L^\infty(\Gamma_j)}  \leq 2^{j(\gamma-d/p)} \max_{\alpham \leq \ell}\| \hat{r}^\alpham \widehat{\partial}^\alpha\hat{v}\|_{L^{\infty}(\widehat{\Gamma})}.
  \end{equation*}
  By the embedding of $W^{2,p}(\widehat{\Gamma})$ in $L^\infty(\widehat{\Gamma})$,
  then, there exists $C>0$ independent of $\ell$ and $j$ such that
  \begin{equation*}
    \max_{\alpham \leq \ell}\|r^{\alpham - \gamma +d/p}\dalpha v \|_{L^\infty(\Gamma_j)}  
    \leq C 2^{j(\gamma-d/p)} \max_{\alpham\leq \ell}\| \hat{r}^\alpham\widehat{\partial}^\alpha\hat{v}\|_{W^{2,p}(\widehat{\Gamma})}.
  \end{equation*}
  Hence, by a simple differentiation, injecting the necessary weight, using
  again that $1/2 < \hat{r}_{|_{\widehat{\Gamma}}}< 1$, and bounding the maximum
  over $\alpham \leq \ell$ with the respective sum, we arrive at
  \begin{equation*}
    \max_{\alpham \leq \ell}\|r^{\alpham - \gamma +d/p}\dalpha v \|_{L^\infty(\Gamma_j)}  
    \leq 
    C 2^{j(\gamma-d/p)} (\ell+1)^2 
    \left(  \sum_{\alpham\leq \ell+2}\| \hat{r}^{\alpham-\gamma}
              \widehat{\partial}^\alpha\hat{v}\|_{L^{p}(\widehat{\Gamma})}^p\right)^{1/p}.
  \end{equation*}
  Scaling back to the original domain, we obtain the existence of $C>0$
  independent of $\ell$ and $j$ such that
  \begin{equation*}
    \max_{\alpham \leq \ell}\|r^{\alpham - \gamma +d/p}\dalpha v \|_{L^\infty(\Gamma_j)}  
    \leq C (\ell+1)^2 \| v\|_{\cK^{\ell+2, p}_\gamma(\Gamma_j)},
  \end{equation*}
  hence there exists $C>0$ such that for all $\ell\in \mathbb{N}$ holds
  \begin{equation*}
    \| v \|_{\mathcal{K}^{\ell,\infty}_{\gamma-d/p} (B_1)}= \sup_{j\in \mathbb{N}_0} \|v\|_{\mathcal{K}^{\ell,\infty}_{\gamma-1} (\Gamma_j)} \leq C  (\ell+1)^2 \| u \|_{\mathcal{K}^{\ell+2, p}_\gamma (B_1)}.
  \end{equation*}
  \end{proof}
  \bibliographystyle{amsalpha-abbrv}
  \bibliography{library}
\end{document}